\documentclass[12pt]{amsart}

\usepackage[margin=1in]{geometry}
\usepackage{amsthm,amssymb,amsmath}

\newtheorem{theorem}{Theorem}[section]
\newtheorem{lemma}[theorem]{Lemma}

\theoremstyle{definition}
\newtheorem{definition}[theorem]{Definition}

\newtheorem{example}[theorem]{Example}
\newtheorem{remark}[theorem]{Remark}

\DeclareMathOperator{\reg}{reg}

\newcommand{\bbP}{\mathbb{P}}

\newcommand{\cA}{\mathcal{A}}
\newcommand{\cB}{\mathcal{B}}
\newcommand{\cC}{\mathcal{C}}
\newcommand{\cF}{\mathcal{F}}
\newcommand{\cI}{\mathcal{I}}
\newcommand{\cO}{\mathcal{O}}

\newcommand{\defining}[1]{\textbf{#1}}

\title[Regularity of bipartite subspace arrangements]
{Castelnuovo--Mumford regularity and arithmetic Cohen--Macaulayness of complete bipartite subspace arrangements}

\author{Zach Teitler}
 \email{zteitler@boisestate.edu}
 \address{Department of Mathematics \\ 1910 University Drive \\ Boise State University \\ Boise, ID 83725-1555 \\ USA}
\author{Douglas A. Torrance}
 \email{dtorrance@monmouthcollege.edu}
 \address{Department of Mathematics \\ University of Idaho \\ P.O.\ Box 441103 \\ Moscow, ID 83844-1103 \\ USA}
 \curraddr{Department of Mathematics and Computer Science \\ Monmouth College \\ 700 E Broadway \\ Monmouth, IL 61462-1963 \\ USA}

\date{\today}



\begin{document}

\begin{abstract}
We give the Castelnuovo--Mumford regularity of arrangements of $(n-2)$-planes in $\bbP^n$
whose incidence graph is a sufficiently large complete bipartite graph,
and determine when such arrangements are arithmetically Cohen--Macaulay.
\end{abstract}

\maketitle

\section{Introduction}
A subspace arrangement $\cA = \{ L_1 , \dotsc , L_d \}$ is a finite collection of linear subspaces $L_i \subset \bbP^n$
with no inclusions $L_i \subset L_j$ for $i \neq j$.
There are many relations among the algebraic properties of the defining ideal $I_\cA$ of the arrangement,
the combinatorial type of the arrangement, and the geometry of the arrangement itself.
See for example the very recent survey article \cite{MR3051389}
on commutative algebra and subspace arrangements.

Following \cite{MR2537025} we consider the \defining{incidence graph} $\Gamma(\cA)$ of a subspace arrangement $\cA$,
defined as the graph with vertex set $\cA$ and an edge $XY$ for $X, Y \in \cA$ if and only if the intersection $X \cap Y$
has greater than expected dimension.
Thus, for example, for an arrangement of lines in $\bbP^3$, the incidence graph simply records which of the lines meet;
for an arrangement of $2$-planes in $\bbP^4$, the incidence graph records which planes meet along lines, and so on.

Plane arrangements whose incidence graph is a Petersen graph are studied in \cite{MR2537025}.
They are shown to link to surfaces with interesting geometric properties such as multisecant lines.
The presence of a multisecant line intersecting a variety $d$ times
indicates a generator of degree at least $d$ in the defining ideal of the variety,
so the variety has Castelnuovo--Mumford regularity at least $d$.
At the same time, for purposes of liaison theory it is natural to study
whether a subspace arrangement is locally or even arithmetically Cohen--Macaulay.

Specifying $\Gamma(\cA)$ usually does not determine the Castelnuovo--Mumford regularity $\reg \cA$,
although it might bound it.
For example, any given finite set of $n$ points in $\bbP^2$, with no three collinear,
can be constructed as a hyperplane section of a line arrangement $\cA$ in $\bbP^3$ having a path of length $n$ as its incidence graph.
The regularity of $\cA$ is equal to the regularity of its hyperplane section (\cite[Prop.~20.20]{eisenbud:comm-alg}),
which is at most $n-1$ but depends on the position of the $n$ points.

We show, however, that when $\cA$ is an arrangement of $(n-2)$-planes in $\bbP^n$
and the incidence graph of $\cA$ is a complete bipartite graph $K_{a,b}$ of type $(a,b)$
then, for sufficiently large values of $a, b$,
the regularity $\reg \cA$ is uniquely determined.

An upper bound on $\reg \cA$ is known.
Indeed, Derksen and Sidman showed in \cite{MR1942401}
that if $\cA$ is an arrangement of linear subspaces,
then $\reg \cA \leq |\cA|$.
Therefore, in the case where $\Gamma(\cA) \cong K_{a,b}$,
we have $\reg I_{\cA} \leq a + b$.
Even better,
Giaimo~\cite{MR2171233} showed that for a reduced connected nondegenerate curve $C \subset \bbP^n$,
$\reg C \leq \deg C - n + 2$
(this generalizes the case of an integral nondegenerate curve,
treated in~\cite{MR704401}).
In our setting, if $\cA \subset \bbP^3$ is a line arrangement with $\Gamma(\cA) \cong K_{a,b}$,
this gives $\reg \cA \leq a+b - 1$.
Our main result shows that for most $a,b$, the actual regularity is lower than these upper bounds.

\begin{theorem}\label{thm: regularity}
Let $\cA$ be an arrangement of $(n-2)$-planes in $\bbP^n$ with incidence graph $\Gamma(\cA) \cong K_{a,b}$,
a complete bipartite graph of type $(a,b)$.
Suppose $a \leq b \leq 2$, $2 \leq a \leq b \leq 3$, or $3 \leq a \leq b$.
Then the defining ideal $I_\cA$ of the arrangement has regularity
$\reg(I_\cA) = \max(a+1,b)$.
\end{theorem}

In addition we determine when these arrangements are arithmetically Cohen--Macaulay.
Again, specifying $\Gamma(\cA)$ usually does not determine whether $\cA$ is arithmetically Cohen--Macaulay;
see \cite[Example 2.5]{MR3003727}.
However when the incidence graph of $\cA$ is a complete bipartite graph we are able to show the following.

\begin{theorem}\label{thm: acm}
Let $\cA$ be an arrangement of $(n-2)$-planes in $\bbP^n$
with $\Gamma(\cA) \cong K_{a,b}$
where $a \leq b \leq 2$, $2 \leq a \leq b \leq 3$, or $3 \leq a \leq b$.
Then $\cA$ is arithmetically Cohen--Macaulay if and only if $b=a$ or $b=a+1$.
\end{theorem}

This was already shown for lines in $\bbP^3$ in \cite{MR778460}, by more algebraic methods.

The idea is to reduce to the case of line arrangements,
then use the special geometry of that setting.
In particular, such a line arrangement will lie on a quadric surface.

For both theorems the omitted values are $a=1$ and $b \geq 3$, or $a=2$ and $b \geq 4$.  
In these cases, we are not able to determine the regularity of $I_{\cA}$
or the arithmetic Cohen--Macaulayness of $\cA$
from $\Gamma(\cA)$ alone, as such arrangements are not guaranteed to reduce to line arrangements on a quadric surface.

Line arrangements on quadric surfaces have appeared in many papers;
in addition to \cite{MR778460,MR3003727} we mention \cite{MR708333},
where these arrangements are used to show that general lines
impose independent conditions on the hypersurfaces containing them.

We use $\cA$ to denote both an arrangement (finite collection of subspaces)
and the projective variety represented by that arrangement (the union of those subspaces).
We assume that the arrangement is defined over the ground field,
in the sense that each subspace in $\cA$ is defined over the field.
Other than this, the field is arbitrary.

\section{Plane arrangements}

For any arrangement $\cA$ with $|\cA| > 1$, the quotient by $\bigcap \cA$ expresses $\cA$
as a cone over an arrangement $\cB$ in a possibly lower-dimensional space.

\begin{lemma}\label{single point}
Let $\cA$ be an arrangement of $(n-2)$-planes in $\bbP^n$.
Suppose $\Gamma(\cA) \cong K_{a,b}$ with $2 \leq a \leq b$.
Then $\bigcap \cA$ is an $(n-4)$-plane.
\end{lemma}

\begin{proof}
By definition, there exist disjoint $\cA_1$ and $\cA_2$
such that $\cA = \cA_1 \cup \cA_2$,
$|\cA_1| = a$, $|\cA_2| = b$,
and, for all $X,Y \in \cA$,
\begin{itemize}
\item if $X \in \cA_1$ and $Y \in \cA_2$, then $\dim(X\cap Y)=n-3$, and
\item if $X,Y \in \cA_i$ for $i=1,2$, then $\dim(X\cap Y)=n-4$.
\end{itemize}

Let $X,Y \in \cA_1$ be distinct and let $Z=X\cap Y$.
Let $U \in \cA_2$.
Since $\dim X \cap Y = n-4$, $U \cap X$ and $U \cap Y$ must be distinct $(n-3)$-planes lying inside $U$,
so $U\cap X$ and $U\cap Y$ intersect in some $(n-4)$-plane.
This $(n-4)$-plane must be $Z$, so $Z\subset U$
and hence $Z \subset \bigcap \cA_2$.
On the other hand, as $|\cA_2| \geq 2$ and any two subspaces in $\cA_2$ intersect in an $(n-4)$-plane,
$\bigcap \cA_2$ has dimension at most $n-4$;
thus $\bigcap \cA_2 = Z = X \cap Y \supseteq \bigcap \cA_1$.

By the same argument, $\bigcap \cA_1 = U \cap V$ for any $U, V \in \bigcap \cA_2$.
It follows that $\bigcap \cA = \bigcap \cA_1 = \bigcap \cA_2 = Z$, an $(n-4)$-plane.
\end{proof}

In the above situation, then, $\cA$ is a cone over an arrangement $\cB$ of lines in $\bbP^3$,
with vertex an $(n-4)$-plane.
We have that $\Gamma(\cB) \cong \Gamma(\cA) \cong K_{a,b}$.
Indeed, if $X, Y \in \cA$ correspond to lines $x, y \in \cB$
(so that $X = x + \bigcap \cA$, $Y = y + \bigcap \cA$ as linear subspaces)
then $X$ and $Y$ are adjacent in $\Gamma(\cA)$ if and only if $x$ and $y$ are adjacent in $\Gamma(\cB)$.

If $a \leq b = 2$ then again $\bigcap \cA$ is an $(n-4)$-plane.
(If $a = b = 2$ the above lemma applies. If $a = 1$, $b = 2$ then $\bigcap \cA = \bigcap \cA_2$ is an $(n-4)$-plane.)
If $a = b = 1$ then $\bigcap \cA$ is an $(n-3)$-plane
and quotienting by any $(n-4)$-dimensional subspace expresses $\cA$ as a cone over
a line arrangement in $\bbP^3$ consisting of two lines through a point.
(While this line arrangement is again a cone and we could go down one dimension further,
we choose to work with line arrangements.)
This proves the following.

\begin{lemma}\label{lem: cone over line arrangement}
If $\cA$ is an arrangement of $(n-2)$-planes in $\bbP^n$
and $\Gamma(\cA) \cong K_{a,b}$ with $a \leq b \leq 2$ or $2 \leq a \leq b$
then $\cA$ is a cone over an arrangement $\cB$ of lines in $\bbP^3$
with $\Gamma(\cB) \cong K_{a,b}$.
\end{lemma}

\begin{lemma}\label{lines lie on quadrics}
If $\cA = \cA_1 \cup \cA_2$ is a complete bipartite arrangement of lines in $\bbP^3$,
$\Gamma(\cA) \cong K_{a,b}$ with $3 \leq a \leq b$ or $a \leq b \leq 3$,
then there is a smooth quadric surface $Q \subset \bbP^3$
such that $\cA$ lies on $Q$.
Specifically, for each $i = 1,2$, the lines of $\cA_i$ lie in one of the two rulings of $Q$.
\end{lemma}

\begin{proof}
Suppose $3 \leq a \leq b$.
Let $X, Y, Z \in \cA_1$ be distinct skew lines.
There is a unique smooth quadric surface $Q$ containing $X, Y, Z$ as lines in one of its
rulings, see \cite[Example 8.36]{Harris}.
(Briefly, $Q$ contains a line $L$ if and only if $Q$ contains three points on $L$;
thus the containment of each of the lines $X, Y, Z$ imposes three conditions on $Q$,
for a total of $9$ conditions in the $10$-dimensional space of quadratic forms on $\bbP^3$.)
Then each line $L \in \cA_2$ lies in $Q$, as it meets $X, Y, Z$,
hence has three points in common with $Q$.
Finally then each line $M \in \cA_1$ lies in $Q$, as it meets each of the lines in $\cA_2$,
giving $b \geq 3$ points in common with $Q$.

That each line in $\cA_1$ meets each line in $\cA_2$ means they lie in opposite rulings of
$Q \cong \bbP^1 \times \bbP^1$;
that each pair of lines in $\cA_1$ (or similarly, $\cA_2$) is skew means they lie in the same ruling.

A similar argument works if $a \leq b = 3$.
The claim is trivial if $a \leq b \leq 2$.
\end{proof}

\section{Regularity}

Recall the following definition.
\begin{definition}
The \defining{(Castelnuovo-Mumford) regularity} of a sheaf $\cF$ on $\bbP^n$ is
\begin{equation*}
\reg \cF = \min \{ d : H^i(\bbP^n,\cF(d-i))=0 \: \forall i>0 \} .
\end{equation*}
For a subvariety $A \subset \mathbb \bbP^n$ we denote by $\reg A$ the regularity $\reg I_A$ of the ideal sheaf of $A$.
In particular, for an arrangement $\cA$ we simply write $\reg \cA$
for the regularity of the defining ideal of the arrangement.

For a comprehensive introduction to this topic, see for example \cite[\textsection 20.5]{eisenbud:comm-alg}
or \cite[\textsection 1.8]{pag1}.
\end{definition}

Suppose $\cA$ is an $(n-2)$-plane arrangement in $\bbP^n$ with $\Gamma(\cA) \cong K_{a,b}$.
As mentioned in the introduction, upper bounds on $\reg \cA$ are known.
We have $\reg \cA \leq a+b$ by a result of Derksen and Sidman \cite{MR1942401},
and indeed $\reg \cA \leq a+b-1$ by a result of Giaimo~\cite{MR2171233},
but these upper bounds, which do not take into account the special geometry
of line arrangements lying on quadric surfaces, are not sharp.

General lower bounds for regularity seem to be less well known.
If $\cA$ is a line arrangement in $\bbP^3$ consisting of $a$ lines in one ruling
of a smooth quadric and $b$ lines in the other ruling,
then $\reg \cA \geq \max\{a,b\}$.
Indeed, a line on the quadric surface meets either $a$ or $b$ of the lines of $\cA$.
Therefore the defining ideal $\cI_\cA$ has a minimal generator in degree at least $\max\{a,b\}$.
(We thank Jessica Sidman for pointing out to us this observation.)
However, even this lower bound, taking into account the special geometry of $\cA$,
is still not sharp.

\begin{theorem} \label{thm: regularity of line arrangement}
Suppose $\cA$ is a line arrangement in $\bbP^3$
such that all lines in $\cA$ lie in a smooth quadric surface $Q$,
$a$ lines of $\cA$ lie in one of the rulings of $Q$,
$b$ lines of $\cA$ lie in the other ruling, and $a \leq b$.
Then $\reg \cA = \max\{a+1,b\}$.
\end{theorem}

\begin{proof}
Let $\cI_\cA \subset \cO_{\bbP^3}$ be the defining ideal sheaf of $\cA$ in $\bbP^3$
and let $\cI_{\cA,Q} \subset \cO_Q$ be the defining ideal sheaf of $\cA$ as a subvariety of $Q$.
We have the exact sequence
\begin{equation*}
0 \rightarrow \cO_{\bbP^3}(-2)
  \overset{\cdot Q}{\longrightarrow} \cI_{\cA}
  \longrightarrow \cI_{\cA,Q}
  \rightarrow 0
\end{equation*}
By hypothesis, $\cI_{\cA,Q}\cong \cO_{\bbP^1\times\bbP^1}(-a,-b)$.  
For $i \geq 1$ and $d \geq i-1$, $H^i(\bbP^3,\cO_{\bbP^3}(d-i-2)) = 0$, so
\begin{equation*}
\begin{split}
H^i(\bbP^3,\cI_{\cA}(d-i))
  &\cong H^i(\bbP^1 \times \bbP^1,\cO_{\bbP^1 \times \bbP^1}(d-i-a,d-i-b)) \\
  &\cong \bigoplus_{j+k=i} H^j(\bbP^1,\cO_{\bbP^1}(d-i-a)) \otimes H^k(\bbP^1,\cO_{\bbP^1}(d-i-b))
\end{split}
\end{equation*}

Now $H^1(\bbP^3,\cI_{\cA}(d-1))=0$
if and only if $d-1-a < 0$ or $d-1-b > -2$, and $d-1-a > -2$ or $d-1-b < 0$;
this is equivalent to $d\leq a$ or $d\geq b$.
And $H^2(\bbP^3,\cI_{\cA}(d-2))=0$
if and only if $d-2-a > -2$ or $d-2-b > -2$;
since $a \leq b$, this is equivalent to
$d\geq a+1$.
\end{proof}

Note that in the case $a=b$, $\cA$ is a complete intersection of type $(2,a)$,
which already implies $\reg \cA = a+1$ \cite[Example~1.8.27]{pag1}.

This result generalizes to higher dimensions,
yielding the statement given in the Introduction.

\begin{proof}[Proof of Theorem \ref{thm: regularity}]
Suppose $a \leq b \leq 2$ or $2 \leq a \leq b \leq 3$ or $3 \leq a \leq b$.
By Lemma~\ref{lem: cone over line arrangement} $\cA$ is a cone over a line arrangement $\cB$ in $\bbP^3$
with $\Gamma(\cB) \cong K_{a,b}$.
By Lemma~\ref{lines lie on quadrics} $\cB$ lies on a quadric, thus satisfies
the hypotheses of Theorem~\ref{thm: regularity of line arrangement}.
Hence $\reg \cB = \max\{a+1,b\}$.
Since $\cA$ is a cone over $\cB$,
we have $\reg \cA = \reg \cB$.
(See, for example, \cite[Prop.~20.20]{eisenbud:comm-alg}: the regularity of $\cA$
is equal to the regularity of its linear section $\cB$.)
\end{proof}

\begin{example}
If $X, Y$ are projective varieties such that $X \cap Y$ is zero-dimensional,
then $\reg X \cup Y \leq \reg X + \reg Y$,
by a result of Caviglia (Corollary~3.4 in \cite{MR2299466}).
We can easily give examples in which equality occurs.
Let $\cA$ be a line arrangement in $\bbP^3$ with $\Gamma(\cA) \cong K_{a_1 , a_2}$,
lying on a smooth quadric surface $Q$.
Let $\cA = \cB \cup \cC$, where $\Gamma(\cB) \cong K_{b_1 , b_2}$,
$\Gamma(\cC) \cong K_{c_1 , c_2}$, $a_i = b_i + c_i$ for $i = 1,2$.
Theorem~\ref{thm: regularity of line arrangement} applies to $\cA$, $\cB$, and $\cC$
since they lie on $Q$;
and the intersection $\cB \cap \cC$ is zero dimensional.
Then $\reg \cA = \reg \cB + \reg \cC$ if and only if one of the following cases occurs:
$b_1 > b_2$ and $c_1 > c_2$ (so $\reg \cA = a_1 = b_1 + c_1 = \reg \cB + \reg \cC$);
or similarly $b_1 < b_2$ and $c_1 < c_2$;
or $b_1 = b_2 + 1$ and $c_2 = c_1 + 1$ (so $a_1 = a_2$, $\reg \cA = a_1 + 1 = b_1 + c_2 = \reg \cB + \reg \cC$);
or similarly $b_2 = b_1 + 1$ and $c_1 = c_2 + 1$.

Cones over these arrangements give examples for which $\reg \cB \cup \cC = \reg \cB + \reg \cC$
while $\cB \cap \cC$ is positive-dimensional.
\end{example}

We can also now prove the condition given in the Introduction determining when
these arrangements are arithmetically Cohen--Macaulay.

\begin{proof}[Proof of Theorem \ref{thm: acm}]
First suppose $\cA$ is a line arrangement in $\bbP^3$ lying on a smooth quadric surface,
with $\Gamma(\cA) \cong K_{a,b}$, $a \leq b$.
The computation in the proof of Theorem~\ref{thm: regularity of line arrangement}
shows that if $a < d < b$, then $H^1(\bbP^3, \cI_\cA(d-1)) \neq 0$.
Thus if $b \geq a+2$ ($\cA$ is ``unbalanced''),
then $\cA$ is not projectively normal
and not arithmetically Cohen--Macaulay.

Conversely, if $b=a$ or $b=a+1$ then the same computation shows
$H^1(\bbP^3, \cI_\cA(d-1)) = 0$ for all $d$,
so $\cA$ is arithmetically Cohen--Macaulay.

In higher dimensions, if $\cA$ is an $(n-2)$-arrangement in $\bbP^n$ with $\Gamma(\cA) \cong K_{a,b}$
with $a \leq b \leq 2$, $2 \leq a \leq b \leq 3$, or $3 \leq a \leq b$,
then $\cA$ is a cone over a line arrangement $\cB \subset \bbP^3$ lying on a quadric with $\Gamma(\cB) \cong K_{a,b}$,
and so $\cA$ is arithmetically Cohen--Macaulay if and only if $\cB$ is (essentially by \cite[Prop.~18.9]{eisenbud:comm-alg}), if and only if $b=a$ or $b=a+1$.
\end{proof}

\begin{remark}
We can do slightly better in $\bbP^3$.
If $\cA$ is a line arrangement in $\bbP^3$ with $\Gamma(\cA) \cong K_{a,b}$
where $a \leq b \leq 3$ or $3 \leq a \leq b$, then
$\cA$ lies on a quadric surface (Lemma~\ref{lines lie on quadrics}),
so $\reg \cA = \max\{a+1,b\}$ (Theorem~\ref{thm: regularity of line arrangement})
and $\cA$ is arithmetically Cohen--Macaulay if and only if $b=a$ or $b=a+1$ (same proof as Theorem~\ref{thm: acm}).

This simply adds the case $(a,b) = (1,3)$ to the list of cases already given in
Theorems~\ref{thm: regularity} and~\ref{thm: acm}.
\end{remark}

\bibliographystyle{amsplain}

\providecommand{\bysame}{\leavevmode\hbox to3em{\hrulefill}\thinspace}
\renewcommand{\MR}[1]{}

\bigskip

\end{document}